\documentclass[12pt]{article}
\usepackage[a4paper,margin=1in]{geometry}
\usepackage{amssymb}
\usepackage{amsmath}
\usepackage{amsfonts}
\usepackage{amsthm}
\usepackage{color}
\usepackage{float}
\usepackage[all]{xy}
\usepackage{scalerel}
\usepackage{mathabx}
\usepackage{delimset}
\usepackage{graphicx} 

\def\E{\mathop{\mbox{\textup{E}}}\nolimits}

\newcommand{\qbinom}{\genfrac{[}{]}{0pt}{}}
\newcommand{\fibonomial}{\genfrac{\{}{\}}{0pt}{}}

\newcommand{\C}{\mathbb{C}}

\newcommand{\N}{\mathbb{N}}

\newcommand{\R}{\mathbb{R}}

\newcommand{\rr}{\mathrm{R}}
\newcommand{\e}{\mathrm{e}}
\newcommand{\RR}{\mathcal{R}}

\newtheorem{theorem}{Theorem}

\newtheorem{definition}{Definition}

\newtheorem{corollary}{Corollary}

\title{\textbf{Interesting Deformed $q$-Series Involving The Central Fibonomial Coefficient}}
\author{Ronald Orozco L\'opez}

\begin{document}

\maketitle

\begin{abstract}
In this paper, we will obtain a variety of interesting $q$-series containing central $q$-binomial coefficients. Our approach is based on manipulating deformed basic hypergeometric series.
\end{abstract}

\section{Introduction}

In his seminal paper \cite{lehmer}, Lehmer called an interesting series in case there is a simple explicit formula for its $n$-th term and at the same time its sum can be expressed in terms of known constants. Some interesting series are:
\begin{align}
    \sum_{n=0}^{\infty}\binom{2n}{n}\frac{(-1)^n}{2^{2n}(2n-1)}x^n&=\sqrt{1+x}.\label{eqn_iden1}\\
    \sum_{n=0}^{\infty}\binom{2n}{n}x^{n}&=\frac{1}{\sqrt{1-4x}}.\label{eqn_iden2}\\
    \sum_{n=0}^{\infty}\binom{2n}{n}\frac{1}{n+1}x^{n}&=\frac{1}{2x}(1-\sqrt{1-4x}).\label{eqn_iden3}\\
    \sum_{n=1}^{\infty}\binom{2n}{n}\frac{1}{n}x^{n}&=2\ln\left(\frac{1-\sqrt{1-4x}}{2x}\right).\label{eqn_iden4}\\
    x\sum_{n=1}^{\infty}\frac{\binom{2n}{n}}{n(n+1)}x^{n}&=2x\ln\left(\frac{1-\sqrt{1-4x}}{x}\right)+\frac{\sqrt{1-4x}}{2}+x(\ln4-1)-\frac{1}{2}.\label{eqn_iden5}\\
    \sum_{n=1}^{\infty}n\binom{2n}{n}x^n&=\frac{2x}{(1-4x)^{3/2}}.\label{eqn_iden6}\\
    \sum_{n=1}^{\infty}n^2\binom{2n}{n}x^n&=\frac{2x(2x+1)}{(1-4x)^{5/2}}.\label{eqn_iden7}\\
    \sum_{n=0}^{\infty}\binom{2n}{n}\frac{1}{2n+1}x^{2n+1}&=\frac{1}{2}\arcsin(2x).\label{eqn_iden8}
\end{align}
The previous series belongs to a class called Lehmer series type I. A second type of Lehmer series is given by

\begin{equation}
    \sum_{n=0}^{\infty}\frac{a_{n}}{\binom{2n}{n}}.
\end{equation}
Much research has been done on these series, especially on the second type, which tends to be more mysterious \cite{boy}, \cite{chen}, \cite{chu}. In this paper, we research deformed $q$-analogues of Lehmer's series Eqs. (\ref{eqn_iden1})-(\ref{eqn_iden8}), i.e., $q$-series containing the central fibonomial coefficients
\begin{equation}
    \fibonomial{2n}{n}_{s,t}=\varphi_{s,t}^{n^2}\qbinom{2n}{n}_{q},
\end{equation}
where $q=\varphi_{s,t}^\prime/\varphi_{s,t}$. Here a $q$-series will be called interesting if its sum can be expressed in terms of deformed basic hypergeometric and $q$-shifted factorial. We will use some $q$-analogues of techniques applied by Lehmer in his paper: integration and the operator $\frac{xd}{dx}$. Among the deformed $q$-analogs, we find those of the Euler, Rogers-Ramanujan, and Exton types. All our deformed $q$-series are representations of deformed basic hypergeometric series ${}_{r}\Phi_{s}$ (See \cite{orozco}).

\section{Preliminaries}

\subsection{$q$-calculus}

All notations and terminologies in this paper for basic hypergeometric series are in \cite{gasper}. The $q$-shifted factorial be defined by
\begin{align*}
    (a;q)_{n}&=\begin{cases}
        1&\text{ if }n=0;\\
    \prod_{k=0}^{n-1}(1-q^{k}a),&\text{ if }n\neq0,\\
    \end{cases}\hspace{1cm}q\in\C,\\
    (a;q)_{\infty}&=\lim_{n\rightarrow\infty}(a;q)_{n}=\prod_{k=0}^{\infty}(1-aq^{k}),\hspace{1cm} \vert q\vert<1.
\end{align*}
The multiple $q$-shifted factorial is defined by
\begin{align*}
    (a_{1},a_{2},\ldots,a_{m};q)_{n}&=(a_{1};q)_{n}(a_{2};q)_{n}\cdots(a_{m};q)_{n},\hspace{0.3cm}q\in\C.
    \end{align*}
In this paper, we will frequently use the following identities:
\begin{align}
    (a;q)_{n}&=\frac{(a;q)_{\infty}}{(aq^n;q)_{\infty}},\label{eqn_iden9}\\
    (a;q)_{n+k}&=(a;q)_{n}(aq^{n};q)_{k},\\
    (a;q)_{2n}&=(a;q^2)_{n}(aq;q)_{n},\\
    (a^2;q^2)_{n}&=(a;q)_{n}(-a;q)_{n},\\
    \frac{1-q}{1-q^{n+1}}&=\frac{(q;q)_{n}}{(q^2;q)_{n}}.
\end{align}
In addition, we will use the identities for binomial coefficients:
\begin{align*}
    \binom{n+k}{2}&=\binom{n}{2}+\binom{k}{2}+nk,\\
    \binom{n-k}{2}&=\binom{n}{2}+\binom{k}{2}+k(1-n).
\end{align*}
The ${}_r\phi_{s}$ basic hypergeometric series is define by
\begin{equation*}
    {}_r\phi_{s}\left(
    \begin{array}{c}
         a_{1},a_{2},\ldots,a_{r} \\
         b_{1},\ldots,b_{s}
    \end{array}
    ;q,z
    \right)=\sum_{n=0}^{\infty}\frac{(a_{1},a_{2},\ldots,a_{r};q)_{n}}{(q;q)_{n}(b_{1},b_{2},\ldots,b_{s};q)_{n}}\Big[(-1)^{n}q^{\binom{n}{2}}\Big]^{1+s-r}z^n.
\end{equation*}
In this paper, we will frequently use the $q$-binomial theorem:
\begin{equation}\label{eqn_qbin_the}
    {}_1\phi_{0}\left(
    \begin{array}{c}
         a\\
         -
    \end{array}
    ;q,z
    \right)=\frac{(az;q)_{\infty}}{(z;q)_{\infty}}=\sum_{n=0}^{\infty}\frac{(a;q)_{n}}{(q;q)_{n}}z^{n}.
\end{equation}
The $q$-differential operator $D_{q}$ is defined by:
\begin{equation*}
    D_{q}f(x)=\frac{f(x)-f(qx)}{x}.
\end{equation*}
The $q$-integral of a function $f(x)$ defined on $[a,b]$ is given by
\begin{equation}
    \int_{a}^{b}f(x)d_{q}x=\int_{0}^{b}f(x)d_{q}x-\int_{0}^{a}f(x)d_{q}x,
\end{equation}
where 
\begin{equation}
    \int_{0}^{a}f(x)d_{q}x=a(1-q)\sum_{n=0}^{\infty}f(aq^{n})q^{n}.
\end{equation}

\section{Deformed basic hypergeometric series}
Orozco \cite{orozco} defined the deformed basic hypergeometric series (DBHS) ${}_{r}\Phi_{s}$ as
    \begin{align*}
        &{}_{r}\Phi_{s}\left(
    \begin{array}{c}
         a_{1},\ldots,a_{r} \\
         b_{1},\ldots,b_{s}
    \end{array}
    ;q,u,z
    \right)\nonumber\\
    &\hspace{3cm}=\sum_{n=0}^{\infty}u^{\binom{n}{2}}\frac{(a_{1},a_{2},\ldots,a_{r};q)_{n}}{(q,b_{1},b_{2},\ldots,b_{s};q)_{n}}[(-1)^nq^{\binom{n}{2}}]^{1+s-r}z^n.                
    \end{align*}
If $u=1$, then ${}_{r}\Phi_{s}={}_{r}\phi_{s}$ and we call this DBHS Euler-I type. If $u=q$, 
\begin{equation}
    {}_{r+1}\Phi_{r}\left(
    \begin{array}{c}
         a_{1},\ldots,a_{r},a_{r+1} \\
         b_{1},\ldots,b_{r}
    \end{array}
    ;q,q,z
    \right)={}_{r+1}\phi_{r+1}\left(
    \begin{array}{c}
         a_{1},\ldots,a_{r+1} \\
         b_{1},\ldots,b_{r},0
    \end{array}
    ;q,-z
    \right),
\end{equation}
for all $z\in\C$, and we call this DBHS Euler-II type.
If $u=q^2$ and mapping $z\mapsto qz$,
\begin{equation}
    {}_{r+1}\Phi_{r}\left(
    \begin{array}{c}
         a_{1},\ldots,a_{r+1} \\
         b_{1},\ldots,b_{r}
    \end{array}
    ;q,q^2,qz
    \right)={}_{r+1}\phi_{r+2}\left(
    \begin{array}{c}
         a_{1},\ldots,a_{r+1} \\
         b_{1},\ldots,b_{r},0,0
    \end{array}
    ;q,qz
    \right)
\end{equation}
for all $z\in\C$, and we call this DBHS Rogers-Ramanujan type. If $u=\sqrt{q}$,
\begin{align}
    &{}_{r+1}\Phi_{r}\left(
    \begin{array}{c}
         a_{1},\ldots,a_{r+1} \\
         b_{1},\ldots,b_{r}
    \end{array}
    ;q,\sqrt{q},z
    \right)\nonumber\\
    &\hspace{2cm}={}_{2r+2}\phi_{2r+2}\left(
    \begin{array}{ccc}
         \sqrt{a_{1}},-\sqrt{a_{1}}&,\ldots,&\sqrt{a_{r+1}},-\sqrt{a_{r+1}}\hspace{0.6cm}\\
         \sqrt{b_{1}},-\sqrt{b_{1}}&,\ldots,&\sqrt{b_{r}},-\sqrt{b_{r}},-\sqrt{q},0
    \end{array}
    ;\sqrt{q},-z
    \right),
\end{align}
for all $z\in\C$, and we call this DBHS Exton type. For all $u\in\C$, define the deformed $q$-exponential function,
\begin{equation*}
    \e_{q}(z,u)=
    \begin{cases}
        \sum_{n=0}^{\infty}u^{\binom{n}{2}}\frac{z^{n}}{(q;q)_{n}}&\text{ if }u\neq0;\\
        1+\frac{z}{1-q}&\text{ if }u=0.
    \end{cases}
\end{equation*}
Some deformed $q$-exponential functions are:
\begin{align*}
    \e_{q}(z,1)&=e_{q}(z)=\frac{1}{(z;q)_{\infty}},\ \vert z\vert<1,\\
    \e_{q}(-z,q)&=\E_{q}(-z)=(z;q)_{\infty},\ z\in\C,\\
    e_{q}(z,\sqrt{q})&=\mathcal{E}_{q}(z)=\sum_{n=0}^{\infty}q^{\frac{1}{2}\binom{n}{2}}\frac{z^n}{(q;q)_{n}}={}_{1}\phi_{1}\left(\begin{array}{c}
         0\\
         -\sqrt{q}
    \end{array};\sqrt{q},-z\right),\ z\in\C,\\
    \e_{q}(qz,q^2)&=\mathcal{R}_{q}(z)=\sum_{n=0}^{\infty}q^{n^2}\frac{z^n}{(q;q)_{n}},\ z\in\C,
\end{align*}
where $\mathcal{E}_{q}(z)$ is the Exton $q$-exponential function and  $\mathcal{R}_{q}(z)$ is the Rogers-Ramanujan function.
The deformed $q$-exponential function has the following representation
\begin{equation*}
    \e_{q}(z,u)={}_{1}\Phi_{0}\left(
    \begin{array}{c}
         0 \\
         -
    \end{array}
    ;q,u,z   
    \right).
\end{equation*}
\begin{theorem}\label{theo_1Phi0}
The DBHS ${}_{1}\Phi_{0}$ has the following representation
\begin{equation}
    {}_{1}\Phi_{0}\left(
    \begin{array}{c}
         a\\
         -
    \end{array}
    ;q,u,x
    \right)=(a;q)_{\infty}\sum_{j=0}^{\infty}\frac{a^j}{(q;q)_{j}}\e_{q}(q^{i}x,u).
\end{equation}
\end{theorem}
\begin{proof}
From Eq.(\ref{eqn_iden9})
  \begin{align*}
        \sum_{k=0}^{\infty}u^{\binom{k}{2}}\frac{(a;q)_{k}}{(q;q)_{k}}x^k&=(a;q)_{\infty}\sum_{k=0}^{\infty}\frac{u^{\binom{k}{2}}}{(q;q)_{k}}\frac{1}{(aq^k;q)_{\infty}}x^k\\
        &=(a;q)_{\infty}\sum_{j=0}^{\infty}\frac{a^j}{(q;q)_{j}}\sum_{k=0}^{\infty}\frac{(q^jx)^ku^{\binom{k}{2}}}{(q;q)_{k}}\\
        &=(a;q)_{\infty}\sum_{j=0}^{\infty}\frac{a^j}{(q;q)_{j}}\e_{q}(q^{j}x,u)
    \end{align*}
as claimed.
\end{proof}
If $u=1$, we obtain the $q$-theorem binomial. We have the following results for $u=q,q^2,\sqrt{q}$.
\begin{corollary}
If $u=q$ in Theorem \ref{theo_1Phi0}, then
\begin{equation}\label{eqn_gAndrews}
        {}_{1}\phi_{1}\left(
    \begin{array}{c}
         a\\
         0
    \end{array}
    ;q,x
    \right)=(-x,a;q)_{\infty}\cdot{}_{2}\phi_{1}\left(
        \begin{array}{c}
             0,0  \\
             -x
        \end{array};
        q, a
        \right).
    \end{equation}
\end{corollary}
\begin{proof}
    \begin{align*}
        {}_{1}\phi_{1}\left(
    \begin{array}{c}
         a\\
         0
    \end{array}
    ;q,x
    \right)&=(a;q)_{\infty}\sum_{j=0}^{\infty}\frac{a^j}{(q;q)_{j}}\e_{q}(q^{i}x,q)\\
    &=(-x,a;q)_{\infty}\sum_{j=0}^{\infty}\frac{a^j}{(q;q)_{j}(-x;q)_{j}}\\
        &=(-x,a;q)_{\infty}{}_{2}\phi_{1}\left(
        \begin{array}{c}
             0,0  \\
             -x
        \end{array};
        q, a
        \right)
    \end{align*}
\end{proof}
Setting $x=q$ in the Eq.(\ref{eqn_gAndrews}), we obtain the identity of Andrews
\begin{equation*}
    \sum_{n=0}^{\infty}\frac{(a;q)_{n}q^{\binom{n+1}{2}}}{(q;q)_{n}}=(-q;q)_{\infty}(aq;q^2)_{\infty}.
\end{equation*}
\begin{corollary}
If $u=q^2$ in Theorem \ref{theo_1Phi0} and mapping $x\mapsto qx$, then
\begin{equation}\label{eqn_gAndrewsq2}
        {}_{1}\phi_{2}\left(
    \begin{array}{c}
         a\\
         0,0
    \end{array}
    ;q,qx
    \right)=(a;q)_{\infty}\sum_{j=0}^{\infty}\frac{a^j}{(q;q)_{j}}\RR_{q}(q^{j}x).
    \end{equation}
\end{corollary}
\begin{corollary}
If $u=\sqrt{q}$ in Theorem \ref{theo_1Phi0}, then
\begin{equation}\label{eqn_gAndrewssqrtq}
        {}_{2}\phi_{2}\left(
    \begin{array}{c}
         \sqrt{a},-\sqrt{a}\\
         -\sqrt{q},0
    \end{array}
    ;\sqrt{q},-x
    \right)=(a;q)_{\infty}\sum_{j=0}^{\infty}\frac{a^j}{(q;q)_{j}}{}_{1}\phi_{1}\left(
    \begin{array}{c}
         0\\
         -\sqrt{q}
    \end{array}
    ;\sqrt{q},-q^jx
    \right).
    \end{equation}
\end{corollary}

\section{The central Fibonomial coefficients}

The generalized Fibonacci polynomials depending on the variables $s,t$ are defined by
\begin{align}
    \brk[c]{0}_{s,t}&=0,\nonumber\\
    \brk[c]{1}_{s,t}&=1,\nonumber\\
    \brk[c]{n+2}_{s,t}&=s\{n+1\}_{s,t}+t\{n\}_{s,t}\label{eqn_def_fibo}.
\end{align}
From Eq.(\ref{eqn_def_fibo}) we obtain the Fibonacci, Pell, Jacobsthal, and Mersenne sequences, among others. The $(s,t)$-Fibonacci constant is the ratio toward which adjacent $(s,t)$-Fibonacci polynomials tend. This is the only positive root of $x^{2}-sx-t=0$. We will let $\varphi_{s,t}$ denote this constant, where
\begin{equation*}
    \varphi_{s,t}=\frac{s+\sqrt{s^{2}+4t}}{2}
\end{equation*}
and
\begin{equation*}
    \varphi_{s,t}^{\prime}=s-\varphi_{s,t}=-\frac{t}{\varphi_{s,t}}=\frac{s-\sqrt{s^{2}+4t}}{2}.
\end{equation*}
The Binet's $(s,t)$-identity is
\begin{equation}\label{eqn_binet}
    \brk[c]{n}_{s,t}=
    \begin{cases}
    \frac{\varphi_{s,t}^{n}-\varphi_{s,t}^{\prime n}}{\varphi_{s,t}-\varphi_{s,t}^{\prime}},&\text{ if }s\neq\pm2i\sqrt{t};\\
    n(\pm i\sqrt{t})^{n-1},&\text{ if }s=\pm2i\sqrt{t}.
    \end{cases}
\end{equation}
The fibonomial coefficients are defined by
\begin{equation}
    \fibonomial{n}{k}_{s,t}=\frac{\brk[c]{n}_{s,t}!}{\brk[c]{k}_{s,t}!\brk[c]{n-k}_{s,t}!}=\varphi_{s,t}^{k(n-k)}\qbinom{n}{k}_{q}=\varphi_{s,t}^{k(n-k)}\frac{(q;q)_{n}}{(q;q)_{k}(q;q)_{n-k}}.\label{eqn_fibo3}
\end{equation}
where $\brk[c]{n}_{s,t}!=\brk[c]{1}_{s,t}\brk[c]{2}_{s,t}\cdots\brk[c]{n}_{s,t}$ is the $(s,t)$-factorial or generalized fibotorial.

For all $\alpha\in\C$ define the generalized Fibonacci functions $\brk[c]{\alpha}_{s,t}$ as
\begin{equation}\label{eqn_gff}
    \brk[c]{\alpha}_{s,t}=\frac{\varphi_{s,t}^{\alpha}-\varphi_{s,t}^{\prime\alpha}}{\varphi_{s,t}-\varphi_{s,t}^\prime}.
\end{equation}
The negative $(s,t)$-Fibonacci functions are given by
\begin{equation}\label{eqn_neg_fibo}
\brk[c]{-\alpha}_{s,t}=-(-t)^{-\alpha}\brk[c]{\alpha}_{s,t}
\end{equation}
for all $\alpha\in\R$.
From Eq.(\ref{eqn_gff}), we obtain
\begin{align*}
    \brk[c]{-\alpha}_{s,t}&=\frac{\varphi_{s,t}^{-\alpha}-\varphi_{s,t}^{\prime(-\alpha)}}{\varphi_{s,t}-\varphi_{s,t}^\prime}\\
    &=-\frac{1}{(-t)^{\alpha}}\frac{\varphi_{s,t}^{\alpha}-\varphi_{s,t}^{\prime\alpha}}{\varphi_{s,t}-\varphi_{s,t}^\prime}\\
    &=-(-t)^{-\alpha}\brk[c]{\alpha}_{s,t}.
\end{align*}
For all $\alpha\in\C$ the generalized fibonomial coefficient is
\begin{align}
    \fibonomial{\alpha}{k}_{s,t}&=\frac{\brk[c]{\alpha}_{s,t}\brk[c]{\alpha-1}_{s,t}\cdots\brk[c]{\alpha-k+1}_{s,t}}{\brk[c]{k}_{s,t}!}.\label{eqn_fibo1}\\
    &=\varphi_{s,t}^{(\alpha-1)k-2\binom{k}{2}}\qbinom{\alpha}{k}_{q}=\frac{(q^{-\alpha};q)_{k}}{(q;q)_{k}}(-\varphi_{s,t}^{\alpha-1}q^{\alpha})^k(\varphi_{s,t}^{2}q)^{-\binom{k}{2}}.\label{eqn_fibo2}
\end{align}
From Eq.(\ref{eqn_fibo3}), the central fibonomial coefficients are
\begin{equation}\label{eqn_stCBC}
    \fibonomial{2n}{n}_{s,t}=\varphi_{s,t}^{n^2}\qbinom{2n}{n}_{q}=\varphi_{s,t}^{n^2}\frac{(\sqrt{q};q)_{n}(-\sqrt{q};q)_{n}(-q;q)_{n}}{(q;q)_{n}}.
\end{equation}
Then we have the following identities for the generalized fibonomial coefficients $\fibonomial{1/2}{n}_{s,t}$ and $\fibonomial{-1/2}{n}_{s,t}$. For all $n\in\N$,
    \begin{align}
        \fibonomial{1/2}{n}_{s,t}&=\qbinom{2n}{n}_{q}\frac{\sqrt{\varphi_{s,t}}^{(-3n+1)}(-1)^{n}(-t)^{-\frac{(n-1)^2}{2}}(1-\sqrt{q})}{(-q;q)_{n}(-\sqrt{q};q)_{n}(1-\sqrt{q}^{2n-1})},\label{eqn1_bino_(1/2)}\\
        &=\frac{(\sqrt{q^{-1}};q)_{n}}{(q;q)_{n}}\left(-\varphi_{s,t}^{-1/2}\sqrt{q}\right)^n(-t)^{-\binom{n}{2}}.\label{eqn3_bino_(1/2)}
    \end{align}
For all $n\in\N$,
    \begin{align}
        \fibonomial{-1/2}{n}_{s,t}&=(-1)^{n}\qbinom{2n}{n}_{q}\frac{\varphi_{s,t}^{-n/2}(-t)^{-n^2/2}}{(-\sqrt{q};q)_{n}(-q;q)_{n}},\label{eqn1_bino_(-1/2)}\\
        &=(-t)^{-n^2/2}\frac{(\sqrt{q};q)_{n}}{(q;q)_{n}}\varphi_{s,t}^{-n/2}(-1)^{n}.\label{eqn2_bino_(-1/2)}
    \end{align}

\section{Deformed Lehmer $q$-series}
\subsection{Definition}

\begin{definition}
A deformed Lehmer $(s,t)$-series of Type I has the form:
\begin{equation}
    \sum_{n=0}^{\infty}a_{n}u^{\binom{n}{2}}\fibonomial{2n}{n}_{s,t}.
\end{equation}
Equally, a deformed Lehmer $(s,t)$-series of Type II is
\begin{equation}
    \sum_{n=0}^{\infty}\frac{a_{n}u^{\binom{n}{2}}}{\fibonomial{2n}{n}_{s,t}}.
\end{equation}
\end{definition}
Set $q=\varphi_{s,t}^\prime/\varphi_{s,t}$. By using Eq.(\ref{eqn_fibo3}) we can to transform the deformed Lehmer $(s,t)$-series in the deformed Lehmer $q$-series:
    \begin{align}
        \sum_{n=0}^{\infty}a_{n}u^{\binom{n}{2}}\fibonomial{2n}{n}_{s,t}&=\sum_{n=0}^{\infty}a_{n}(\varphi_{s,t}^{2}u)^{\binom{n}{2}}\qbinom{2n}{n}_{q}\varphi_{s,t}^{n}\\
        &=\sum_{n=0}^{\infty}a_{n}(\varphi_{s,t}^2u)^{\binom{n}{2}}\frac{(\sqrt{q};q)_{n}(-\sqrt{q};q)_{n}(-q;q)_{n}}{(q;q)_{n}}\varphi_{s,t}^{n},\\
        \sum_{n=0}^{\infty}\frac{a_{n}u^{\binom{n}{2}}}{\fibonomial{2n}{n}_{s,t}}&=\sum_{n=0}^{\infty}\frac{a_{n}(\varphi_{s,t}^{-2}u)^{\binom{n}{2}}}{\qbinom{2n}{n}_{q}}\varphi_{s,t}^{-n}\\
        &=\sum_{n=0}^{\infty}a_{n}(\varphi_{s,t}^{-2}u)^{\binom{n}{2}}\frac{(q;q)_{n}}{(\sqrt{q};q)_{n}(-\sqrt{q};q)_{n}(-q;q)_{n}}\varphi_{s,t}^{-n}.
    \end{align}
The following function will be very relevant for the rest of the paper.
\begin{definition}
For all $\alpha\in\C$, we define the $u$-deformed $(s,t)$-binomial functions as
\begin{equation}\label{eqn_nbs}
    \rr_{\alpha}(x;u|s,t)=\sum_{n=0}^{\infty}\fibonomial{\alpha}{n}_{s,t}u^{\binom{n}{2}}x^n.
\end{equation}
\end{definition}

\begin{theorem}\label{theo_rep_dbhs}
For all $\alpha\in\C$ the representation of $\rr_{\alpha}(x;u|s,t)$ in DBHS form is
    \begin{equation*}
        \rr_{\alpha}(x;u|s,t)={}_{1}\Phi_{0}\left(
        \begin{array}{c}
             q^{-\alpha}\\
             -
        \end{array};
        q,-\frac{u}{t},-\varphi_{s,t}^{\alpha-1}q^{\alpha}x
        \right).
    \end{equation*}
\end{theorem}
\begin{proof}
From previous definition and Eqs.(\ref{eqn_fibo2}),
\begin{align*}
    \rr_{\alpha}(x;u|s,t)&=\sum_{n=0}^{\infty}\fibonomial{\alpha}{n}_{s,t}u^{\binom{n}{2}}x^n\\
    &=\sum_{n=0}^{\infty}\frac{(q^{-\alpha};q)_{n}}{(q;q)_{n}}(-\varphi_{s,t}^{\alpha-1}q^{\alpha})^n(\varphi_{s,t}^{2}q)^{-\binom{n}{2}}u^{\binom{n}{2}}x^n\\
    &=\sum_{n=0}^{\infty}(-u/t)^{\binom{n}{2}}\frac{(q^{\alpha};q)_{n}}{(q;q)_{n}}(-\varphi_{s,t}^{\alpha-1}q^{\alpha}x)^n\\
    &={}_{1}\Phi_{0}\left(
        \begin{array}{c}
             q^{-\alpha}\\
             -
        \end{array};
        q,-\frac{u}{t},-\varphi_{s,t}^{\alpha-1}q^{\alpha}x
        \right).    
\end{align*}    
\end{proof}
From Theorems \ref{theo_1Phi0} and \ref{theo_rep_dbhs}, we have the following representation for the function $\rr_{\alpha}(x;u|s,t)$.
\begin{theorem}
For all complex number $\alpha$ with $\alpha\neq n$,
    \begin{equation}
        \rr_{\alpha}(x;u|s,t)=(q^{-\alpha};q)_{\infty}\sum_{n=0}^{\infty}\frac{q^{-\alpha n}}{(q;q)_{n}}\e_{q}\left(-\varphi_{s,t}^{\alpha-1}q^{\alpha+n}x,-\frac{u}{t}\right).
    \end{equation}
\end{theorem}

\subsection{Deformed $q$-analog of Lehmer series}

\begin{theorem}\label{theo_dfg_iden1}
The deformed $q$-analog of Eq.(\ref{eqn_iden1}) is
\begin{multline}
    \sum_{n=0}^{\infty}\qbinom{2n}{n}_{q}\frac{(-1)^{n}(-u/t)^{\binom{n}{2}}}{(-q;q)_{n}(-\sqrt{q};q)_{n}(1-\sqrt{q}^{2n-1})}(\varphi_{s,t}^{-3/2}\sqrt{-t}x)^n\\
    =\frac{\sqrt{-t}}{(1-\sqrt{q})\sqrt{\varphi_{s,t}}}{}_{1}\Phi_{0}\left(
        \begin{array}{c}
             \sqrt{q^{-1}}\\
             -
        \end{array};
        q,-\frac{u}{t},-\varphi_{s,t}^{-1/2}\sqrt{q}x
        \right).\label{eqn_dfg_iden1}
\end{multline}    
\end{theorem}
\begin{proof}
On the one side, from Eq.(\ref{eqn1_bino_(1/2)}) we have
\begin{align*}
    \rr_{1/2}(x;u|s,t)&=\sum_{n=0}^{\infty}\fibonomial{1/2}{n}_{s,t}u^{\binom{n}{2}}x^n\\
    &=\sum_{n=0}^{\infty}\qbinom{2n}{n}_{q}\frac{\sqrt{\varphi_{s,t}}^{(-3n+1)}(-1)^{n}(-t)^{-\frac{(n-1)^2}{2}}(1-\sqrt{q})}{(-q;q)_{n}(-\sqrt{q};q)_{n}(1-\sqrt{q}^{2n-1})}u^{\binom{n}{2}}x^n\\
    &=\frac{\sqrt{\varphi_{s,t}}(1-\sqrt{q})}{\sqrt{-t}}\sum_{n=0}^{\infty}\qbinom{2n}{n}_{q}\frac{(-1)^{n}(-u/t)^{\binom{n}{2}}}{(-q;q)_{n}(-\sqrt{q};q)_{n}(1-\sqrt{q}^{2n-1})}(\varphi_{s,t}^{-3/2}\sqrt{-t}x)^n.
\end{align*}
On the other hand, from Theorem \ref{theo_rep_dbhs}
\begin{align*}
    \rr_{1/2}(x;u|s,t)={}_{1}\Phi_{0}\left(
        \begin{array}{c}
             q^{-1/2}\\
             -
        \end{array};
        q,-\frac{u}{t},-\varphi_{s,t}^{-1/2}q^{1/2}x
        \right).
\end{align*}

\end{proof}

\begin{theorem}\label{theo_dfg_iden2}
The deformed $q$-analog of Eq.(\ref{eqn_iden2}) is
\begin{equation}
    \sum_{n=0}^{\infty}\qbinom{2n}{n}_{q}\frac{(-u/t)^{\binom{n}{2}}}{(-\sqrt{q};q)_{n}(-q;q)_{n}}\left(\frac{4x}{\sqrt{-t\varphi_{s,t}}}\right)^n\\
    ={}_{1}\Phi_{0}\left(
        \begin{array}{c}
             \sqrt{q}\\
             -
        \end{array};
        q,-\frac{u}{t},\frac{4x}{\sqrt{-t\varphi_{s,t}}}
        \right).\label{eqn_dfg_iden2}
\end{equation}
\end{theorem}
\begin{proof}
On the one side, from Eq.(\ref{eqn1_bino_(-1/2)})
\begin{align*}
    \rr_{-1/2}(-4x;u|s,t)&=\sum_{n=0}^{\infty}\fibonomial{-1/2}{n}_{s,t}u^{\binom{n}{2}}(-4x)^{n}\\
    &=\sum_{n=0}^{\infty}(-1)^{n}\qbinom{2n}{n}_{q}\frac{\varphi_{s,t}^{-n/2}(-t)^{-n^2/2}}{(-\sqrt{q};q)_{n}(-q;q)_{n}}u^{\binom{n}{2}}(-4x)^{n}\\
    &=\sum_{n=0}^{\infty}\qbinom{2n}{n}_{q}\frac{(-u/t)^{\binom{n}{2}}}{(-\sqrt{q};q)_{n}(-q;q)_{n}}\left(\frac{4x}{\sqrt{-t\varphi_{s,t}}}\right).
\end{align*}
On the other hand, from Theorem \ref{theo_rep_dbhs}
\begin{align*}
    \rr_{-1/2}(-4x,u|s,t)={}_{1}\Phi_{0}\left(
        \begin{array}{c}
             \sqrt{q}\\
             -
        \end{array};
        q,-\frac{u}{t},\frac{4x}{\sqrt{-t\varphi_{s,t}}}
        \right).
\end{align*}

\end{proof}
\begin{theorem}\label{theo_dfg_iden3}
The deformed $q$-analog of Eq.(\ref{eqn_iden3}) is
\begin{multline}
    \sum_{n=0}^{\infty}\qbinom{2n}{n}_{q}\frac{(-u/t)^{\binom{n}{2}}}{(-\sqrt{q};q)_{n}(-q;q)_{n}(1-q^{n+1})}\left(\frac{4x}{\sqrt{-t\varphi_{s,t}}}\right)^n\\
    =\frac{1}{1-q}\cdot{}_{2}\Phi_{1}\left(
        \begin{array}{c}
             \sqrt{q},q\\
             q^2
        \end{array};
        q,-\frac{u}{t},\frac{4x}{\sqrt{-t\varphi_{s,t}}}
        \right).\label{eqn_dfg_iden3}
\end{multline}
\end{theorem}
\begin{proof}
By applying the $q$-integral operator to Eq.(\ref{eqn_dfg_iden2}) from $0$ to $x$ and then dividing both sides by $x$, we obtain on the one side
    \begin{align*}
        &\frac{1}{x}\int_{0}^{x}\sum_{n=0}^{\infty}\qbinom{2n}{n}_{q}\frac{(-u/t)^{\binom{n}{2}}}{(-\sqrt{q};q)_{n}(-q;q)_{n}}\left(\frac{4}{\sqrt{-t\varphi_{s,t}}}\right)^n\theta^n d_{q}\theta\\
    &\hspace{2cm}=\sum_{n=0}^{\infty}\qbinom{2n}{n}_{q}\frac{(-u/t)^{\binom{n}{2}}}{(-\sqrt{q};q)_{n}(-q;q)_{n}(1-q^{n+1})}\left(\frac{4x}{\sqrt{-t\varphi_{s,t}}}\right)^n, 
    \end{align*} 
and on the other hand
\begin{align*}
    &\frac{1}{x}\int_{0}^{x}{}_{1}\Phi_{0}\left(
        \begin{array}{c}
             \sqrt{q}\\
             -
        \end{array};
        q,-\frac{u}{t},\frac{4\theta}{\sqrt{-t\varphi_{s,t}}}
        \right)d_{q}\theta\\
        &\hspace{2cm}=\frac{1}{x}\int_{0}^{x}\sum_{n=0}^{\infty}\frac{(\sqrt{q};q)_{n}(-u/t)^{\binom{n}{2}}}{(q;q)_{n}}\left(\frac{4}{\sqrt{-t\varphi_{s,t}}}\right)^n\theta^{n}d_{q}\theta\\
        &\hspace{2cm}=\sum_{n=0}^{\infty}\frac{(\sqrt{q};q)_{n}(-u/t)^{\binom{n}{2}}}{(q;q)_{n}(1-q^{n+1})}\left(\frac{4}{\sqrt{-t\varphi_{s,t}}}\right)^nx^{n}\\
        &\hspace{2cm}=\frac{1}{1-q}\sum_{n=0}^{\infty}\frac{(\sqrt{q};q)_{n}(-u/t)^{\binom{n}{2}}}{(q^2;q)_{n}}\left(\frac{4x}{\sqrt{-t\varphi_{s,t}}}\right)^n\\
        &\hspace{2cm}={}_{2}\Phi_{1}\left(
        \begin{array}{c}
             \sqrt{q},q\\
             q^2
        \end{array};
        q,-\frac{u}{t},\frac{4x}{\sqrt{-t\varphi_{s,t}}}
        \right).
\end{align*}
\end{proof}

\begin{theorem}\label{theo_dfg_iden4}
The deformed $q$-analog of Eq.(\ref{eqn_iden4}) is
\begin{multline}
    \sum_{n=1}^{\infty}\qbinom{2n}{n}_{q}\frac{(-u/t)^{\binom{n}{2}}}{(-\sqrt{q};q)_{n}(-q;q)_{n}(1-q^n)}\left(\frac{4}{\sqrt{-t\varphi_{s,t}}}\right)^{n}x^{n}\\
    =\frac{4(1-\sqrt{q})x}{(1-q)^2\sqrt{-t\varphi_{s,t}}}{}_{3}\Phi_{2}\left(
        \begin{array}{c}
             q\sqrt{q},q,q\\
             q^2,q^2
        \end{array};
        q,-\frac{u}{t},\frac{-4ux}{t\sqrt{-t\varphi_{s,t}}}
        \right).\label{eqn_dfg_iden4}
\end{multline}
\end{theorem}
\begin{proof}
By transposing its first term to the right side, dividing both sides by $x$, and then $q$-integrating, we have on the one side
    \begin{align*}
        &\int_{0}^{x}\sum_{n=1}^{\infty}\qbinom{2n}{n}_{q}\frac{(-u/t)^{\binom{n}{2}}}{(-\sqrt{q};q)_{n}(-q;q)_{n}}\left(\frac{4}{\sqrt{-t\varphi_{s,t}}}\right)^{n}\theta^{n-1}d_{q}\theta\\
        &\hspace{1cm}=\sum_{n=1}^{\infty}\qbinom{2n}{n}_{q}\frac{(-u/t)^{\binom{n}{2}}}{(-\sqrt{q};q)_{n}(-q;q)_{n}(1-q^n)}\left(\frac{4}{\sqrt{-t\varphi_{s,t}}}\right)^{n}x^{n}.
    \end{align*}
and on the other hand    
    \begin{align*}
        &\int_{0}^{x}\bigg[\frac{1}{\theta}\cdot{}_{1}\Phi_{0}\left(
        \begin{array}{c}
             \sqrt{q}\\
             -
        \end{array};
        q,-\frac{u}{t},\frac{4\theta}{\sqrt{-t\varphi_{s,t}}}
        \right)-\frac{1}{\theta}\bigg]d_{q}\theta\\
        &\hspace{1cm}=\int_{0}^{x}\sum_{n=1}^{\infty}\frac{(\sqrt{q};q)_{n}(-u/t)^{\binom{n}{2}}}{(q;q)_{n}}\left(\frac{4}{\sqrt{-t\varphi_{s,t}}}\right)^{n}\theta^{n-1}d_{q}\theta\\
        &\hspace{1cm}=\sum_{n=1}^{\infty}\frac{(\sqrt{q};q)_{n}(-u/t)^{\binom{n}{2}}}{(q;q)_{n}(1-q^n)}\left(\frac{4x}{\sqrt{-t\varphi_{s,t}}}\right)^{n}\\
        &\hspace{1cm}=\frac{1}{1-q}\sum_{n=1}^{\infty}\frac{(\sqrt{q};q)_{n}(q;q)_{n-1}(-u/t)^{\binom{n}{2}}}{(q;q)_{n}(q^2;q)_{n-1}}\left(\frac{4x}{\sqrt{-t\varphi_{s,t}}}\right)^{n}\\
        &\hspace{1cm}=\frac{1}{1-q}\sum_{n=0}^{\infty}\frac{(\sqrt{q};q)_{n+1}(q;q)_{n}(-u/t)^{\binom{n+1}{2}}}{(q;q)_{n+1}(q^2;q)_{n}}\left(\frac{4x}{\sqrt{-t\varphi_{s,t}}}\right)^{n+1}\\
        &\hspace{1cm}=\frac{4(1-\sqrt{q})x}{(1-q)^2\sqrt{-t\varphi_{s,t}}}\sum_{n=0}^{\infty}\frac{(q\sqrt{q};q)_{n}(q;q)_{n}(-u/t)^{\binom{n}{2}}}{(q^2;q)_{n}(q^2;q)_{n}}\left(\frac{-4ux}{t\sqrt{-t\varphi_{s,t}}}\right)^{n}\\
        &\hspace{1cm}=\frac{4(1-\sqrt{q})x}{(1-q)^2\sqrt{-t\varphi_{s,t}}}{}_{3}\Phi_{2}\left(
        \begin{array}{c}
             q\sqrt{q},q,q\\
             q^2,q^2
        \end{array};
        q,-\frac{u}{t},\frac{-4ux}{t\sqrt{-t\varphi_{s,t}}}
        \right).
    \end{align*}
\end{proof}

\begin{theorem}\label{theo_dfg_iden5}
The deformed $q$-analog of Eq.(\ref{eqn_iden5}) is
\begin{multline}
    \sum_{n=1}^{\infty}\qbinom{2n}{n}_{q}\frac{(-u/t)^{\binom{n}{2}}}{(-\sqrt{q};q)_{n}(-q;q)_{n}(1-q^{n})(1-q^{n+1})}\left(\frac{4x}{\sqrt{-t\varphi_{s,t}}}\right)^n\\
    =\frac{4(1-\sqrt{q})x}{(1-q)^2(1+q)\sqrt{-t\varphi_{s,t}}}{}_{3}\Phi_{2}\left(
        \begin{array}{c}
             q\sqrt{q},q,q\\
             q^2,q^3
        \end{array};
        q,-\frac{u}{t},\frac{-4ux}{t\sqrt{-t\varphi_{s,t}}}
        \right).\label{eqn_dfg_iden5}
\end{multline}
\end{theorem}
\begin{proof}
By $q$-integrate Eq.(\ref{eqn_dfg_iden4}), we have on the one side
\begin{align*}
    &\frac{4(1-\sqrt{q})}{(1-q)^2\sqrt{-t\varphi_{s,t}}}\int_{0}^{x}\theta\cdot{}_{3}\Phi_{2}\left(
        \begin{array}{c}
             q\sqrt{q},q,q\\
             q^2,q^2
        \end{array};
        q,-\frac{u}{t},\frac{-4u\theta}{t\sqrt{-t\varphi_{s,t}}}
        \right)d_{q}\theta\\
        &\hspace{0.5cm}=\frac{4(1-\sqrt{q})}{(1-q)^2\sqrt{-t\varphi_{s,t}}}\int_{0}^{x}\sum_{n=0}^{\infty}\frac{(q\sqrt{q};q)_{n}(q;q)_{n}(-u/t)^{\binom{n}{2}}}{(q^2;q)_{n}(q^2;q)_{n}}\left(\frac{-4u}{t\sqrt{-t\varphi_{s,t}}}\right)^{n}\theta^{n+1}d_{q}\theta\\
        &\hspace{0.5cm}=\frac{4(1-\sqrt{q})}{(1-q)^2\sqrt{-t\varphi_{s,t}}}\sum_{n=0}^{\infty}\frac{(q\sqrt{q};q)_{n}(q;q)_{n}(-u/t)^{\binom{n}{2}}}{(q^2;q)_{n}(q^2;q)_{n}(1-q^{n+2})}\left(\frac{-4u}{t\sqrt{-t\varphi_{s,t}}}\right)^{n}x^{n+2}d_{q}\theta\\
        &\hspace{0.5cm}=\frac{4(1-\sqrt{q})x^2}{(1-q)^2(1+q)\sqrt{-t\varphi_{s,t}}}\sum_{n=0}^{\infty}\frac{(q\sqrt{q};q)_{n}(q;q)_{n}(q^2;q)_{n}(-u/t)^{\binom{n}{2}}}{(q^2;q)_{n}(q^2;q)_{n}(q^3;q)_{n}}\left(\frac{-4u}{t\sqrt{-t\varphi_{s,t}}}\right)^{n}x^{n}d_{q}\theta\\
        &\hspace{0.5cm}=\frac{4(1-\sqrt{q})x^2}{(1-q)^2(1+q)\sqrt{-t\varphi_{s,t}}}{}_{3}\Phi_{2}\left(
        \begin{array}{c}
             q\sqrt{q},q,q\\
             q^2,q^3
        \end{array};
        q,-\frac{u}{t},\frac{-4u\theta}{t\sqrt{-t\varphi_{s,t}}}
        \right).
\end{align*}    
and on the other hand
\begin{align*}
    &\int_{0}^{x}\sum_{n=1}^{\infty}\qbinom{2n}{n}_{q}\frac{(-u/t)^{\binom{n}{2}}}{(-\sqrt{q};q)_{n}(-q;q)_{n}(1-q^n)}\left(\frac{4}{\sqrt{-t\varphi_{s,t}}}\right)^{n}\theta^{n}d_{q}\theta\\
    &\hspace{1cm}=\sum_{n=1}^{\infty}\qbinom{2n}{n}_{q}\frac{(-u/t)^{\binom{n}{2}}}{(-\sqrt{q};q)_{n}(-q;q)_{n}(1-q^n)(1-q^{n+1})}\left(\frac{4}{\sqrt{-t\varphi_{s,t}}}\right)^{n}x^{n+1}.
\end{align*}
\end{proof}

\begin{theorem}\label{theo_dfg_iden6}
The deformed $q$-analog of Eq.(\ref{eqn_iden6}) is
\begin{multline}
    \sum_{n=1}^{\infty}\qbinom{2n}{n}_{q}\frac{(1-q^{n})(-u/t)^{\binom{n}{2}}}{(-\sqrt{q};q)_{n}(-q;q)_{n}}\left(\frac{4x}{\sqrt{-t\varphi_{s,t}}}\right)^n\\
    =\frac{4(1-\sqrt{q})x}{\sqrt{-t\varphi_{s,t}}}{}_{1}\Phi_{0}\left(
        \begin{array}{c}
             q\sqrt{q}\\
             -
        \end{array};
        q,-\frac{u}{t},\frac{-4ux}{t\sqrt{-t\varphi_{s,t}}}
        \right).\label{eqn_dfg_iden6}
\end{multline}
\end{theorem}
\begin{proof}
By applying the operator $xD_{q}$ to Eq.(\ref{eqn_dfg_iden2}), we have on the one side
    \begin{align*}
        &xD_{q}\left\{\sum_{n=0}^{\infty}\qbinom{2n}{n}_{q}\frac{(-u/t)^{\binom{n}{2}}}{(-\sqrt{q};q)_{n}(-q;q)_{n}}\left(\frac{4}{\sqrt{-t\varphi_{s,t}}}\right)^nx^n\right\}\\
        &\hspace{2cm}=\sum_{n=1}^{\infty}\qbinom{2n}{n}_{q}\frac{(1-q^{n})(-u/t)^{\binom{n}{2}}}{(-\sqrt{q};q)_{n}(-q;q)_{n}}\left(\frac{4x}{\sqrt{-t\varphi_{s,t}}}\right)^n
    \end{align*}
and the other hand
    \begin{align*}
        &xD_{q}\left\{{}_{1}\Phi_{0}\left(
        \begin{array}{c}
             \sqrt{q}\\
             -
        \end{array};
        q,-\frac{u}{t},\frac{4x}{\sqrt{-t\varphi_{s,t}}}
        \right)\right\}\\
        &\hspace{2cm}=xD_{q}\left\{\sum_{n=0}^{\infty}\frac{(\sqrt{q};q)_{n}(-u/t)^{\binom{n}{2}}}{(q;q)_{n}}\left(\frac{4}{\sqrt{-t\varphi_{s,t}}}\right)^nx^n\right\}\\
        &\hspace{2cm}=x\sum_{n=1}^{\infty}\frac{(\sqrt{q};q)_{n}(1-q^n)(-u/t)^{\binom{n}{2}}}{(q;q)_{n}}\left(\frac{4}{\sqrt{-t\varphi_{s,t}}}\right)^nx^{n-1}\\
        &\hspace{2cm}=\frac{4(1-\sqrt{q})}{\sqrt{-t\varphi_{s,t}}}x\sum_{n=0}^{\infty}\frac{(q\sqrt{q};q)_{n}(-u/t)^{\binom{n}{2}}}{(q;q)_{n}}\left(\frac{-4ux}{t\sqrt{-t\varphi_{s,t}}}\right)^n\\
        &\hspace{2cm}=\frac{4(1-\sqrt{q})x}{\sqrt{-t\varphi_{s,t}}}{}_{1}\Phi_{0}\left(
        \begin{array}{c}
             q\sqrt{q}\\
             -
        \end{array};
        q,-\frac{u}{t},\frac{-4ux}{t\sqrt{-t\varphi_{s,t}}}
        \right).
    \end{align*}
\end{proof}

\begin{theorem}\label{theo_dfg_iden7}
The deformed $q$-analog of Eq.(\ref{eqn_iden7}) is
\begin{multline}
    \sum_{n=1}^{\infty}\qbinom{2n}{n}_{q}\frac{(1-q^{n})^2(-u/t)^{\binom{n}{2}}}{(-\sqrt{q};q)_{n}(-q;q)_{n}}\left(\frac{4x}{\sqrt{-t\varphi_{s,t}}}\right)^n\\
    =\frac{4(1-\sqrt{q})x}{\sqrt{-t\varphi_{s,t}}}{}_{1}\Phi_{0}\left(
        \begin{array}{c}
             q\sqrt{q}\\
             -
        \end{array};
        q,-\frac{u}{t},\frac{-4qux}{t\sqrt{-t\varphi_{s,t}}}
        \right)\\
        +\frac{4^2(\sqrt{q};q)_{2}ux^2}{t^2\varphi_{s,t}}{}_{1}\Phi_{0}\left(
        \begin{array}{c}
             q^2\sqrt{q}\\
             -
        \end{array};
        q,-\frac{u}{t},\frac{-4u^2x}{t^2\sqrt{-t\varphi_{s,t}}}
        \right).\label{eqn_dfg_iden7}
\end{multline}
\end{theorem}
\begin{proof}
By applying the operator $xD_{q}$ to Eq.(\ref{eqn_dfg_iden6}), we have on the one side
\begin{align*}
    &xD_{q}\left\{\sum_{n=1}^{\infty}\qbinom{2n}{n}_{q}\frac{(1-q^{n})(-u/t)^{\binom{n}{2}}}{(-\sqrt{q};q)_{n}(-q;q)_{n}}\left(\frac{4x}{\sqrt{-t\varphi_{s,t}}}\right)^n\right\}\\
    &\hspace{1cm}=\sum_{n=1}^{\infty}\qbinom{2n}{n}_{q}\frac{(1-q^{n})^2(-u/t)^{\binom{n}{2}}}{(-\sqrt{q};q)_{n}(-q;q)_{n}}\left(\frac{4x}{\sqrt{-t\varphi_{s,t}}}\right)^n
\end{align*}
and on the other hand
\begin{align*}
    &xD_{q}\left\{\frac{4(1-\sqrt{q})x}{\sqrt{-t\varphi_{s,t}}}{}_{1}\Phi_{0}\left(
        \begin{array}{c}
             q\sqrt{q}\\
             -
        \end{array};
        q,-\frac{u}{t},\frac{-4ux}{t\sqrt{-t\varphi_{s,t}}}
        \right)\right\}\\
        &\hspace{1cm}=\frac{4(1-\sqrt{q})x}{\sqrt{-t\varphi_{s,t}}}{}_{1}\Phi_{0}\left(
        \begin{array}{c}
             q\sqrt{q}\\
             -
        \end{array};
        q,-\frac{u}{t},\frac{-4qux}{t\sqrt{-t\varphi_{s,t}}}
        \right)\\
        &\hspace{2cm}+\frac{4^2(\sqrt{q};q)_{2}ux^2}{t^2\varphi_{s,t}}{}_{1}\Phi_{0}\left(
        \begin{array}{c}
             q^2\sqrt{q}\\
             -
        \end{array};
        q,-\frac{u}{t},\frac{-4u^2x}{t^2\sqrt{-t\varphi_{s,t}}}
        \right).
\end{align*}
\end{proof}

\begin{theorem}\label{theo_dfg_iden8}
The deformed $q$-analog of Eq.(\ref{eqn_iden8}) is
\begin{multline}
    \sum_{n=1}^{\infty}\qbinom{2n}{n}_{q}\frac{(-u/t)^{\binom{n}{2}}}{(-\sqrt{q};q)_{n}(-q;q)_{n}(1-q^{2n+1})}\left(\frac{4x^2}{\sqrt{-t\varphi_{s,t}}}\right)^n\\
    =\frac{x}{1-q}{}_{3}\Phi_{2}\left(
        \begin{array}{c}
             \sqrt{q},\sqrt{q},-\sqrt{q}\\
             q\sqrt{q},-q\sqrt{q}
        \end{array};
        q,-\frac{u}{t},\frac{-4x^2}{t\sqrt{\varphi_{s,t}}}
        \right).\label{eqn_dfg_iden8}
\end{multline}
\end{theorem}
\begin{proof}
If in Eq.(\ref{eqn_dfg_iden1}) we replace $x$ by $x^2$ and then $q$-integrate both sides, we get
    \begin{align*}
        &\int_{0}^{x}\sum_{n=0}^{\infty}\qbinom{2n}{n}_{q}\frac{(-u/t)^{\binom{n}{2}}}{(-\sqrt{q};q)_{n}(-q;q)_{n}}\left(\frac{4}{\sqrt{-t\varphi_{s,t}}}\right)^n\theta^{2n}d_{q}\theta\\
        &\hspace{1cm}=\sum_{n=0}^{\infty}\qbinom{2n}{n}_{q}\frac{(-u/t)^{\binom{n}{2}}}{(-\sqrt{q};q)_{n}(-q;q)_{n}(1-q^{2n+1})}\left(\frac{4}{\sqrt{-t\varphi_{s,t}}}\right)^nx^{2n+1}
    \end{align*}
and
    \begin{align*}
        &\int_{0}^{x}{}_{1}\Phi_{0}\left(
        \begin{array}{c}
             \sqrt{q}\\
             -
        \end{array};
        q,-\frac{u}{t},\frac{4\theta^2}{\sqrt{\varphi_{s,t}}}
        \right)d_{q}\theta\\
        &\hspace{1cm}=\int_{0}^{x}\sum_{n=0}^{\infty}\frac{(\sqrt{q};q)_{n}(-u/t)^{\binom{n}{2}}}{(q;q)_{n}}\left(\frac{4}{\sqrt{-t\varphi_{s,t}}}\right)^{n}\theta^{2n}d_{q}\theta\\
        &\hspace{1cm}=\sum_{n=0}^{\infty}\frac{(\sqrt{q};q)_{n}(-u/t)^{\binom{n}{2}}}{(q;q)_{n}(1-q^{2n+1})}\left(\frac{4}{\sqrt{-t\varphi_{s,t}}}\right)^{n}x^{2n+1}\\
        &\hspace{1cm}=\frac{1}{1-q}\sum_{n=0}^{\infty}\frac{(\sqrt{q};q)_{n}(\sqrt{q};q)_{n}(-\sqrt{q};q)_{n}(-u/t)^{\binom{n}{2}}}{(q\sqrt{q};q)_{n}(-q\sqrt{q};q)_{n}(q;q)_{n}}\left(\frac{4}{\sqrt{-t\varphi_{s,t}}}\right)^{n}x^{2n+1}\\
        &\hspace{1cm}=\frac{x}{1-q}{}_{3}\Phi_{2}\left(
        \begin{array}{c}
             \sqrt{q},\sqrt{q},-\sqrt{q}\\
             q\sqrt{q},-q\sqrt{q}
        \end{array};
        q,-\frac{u}{t},\frac{-4x^2}{t\sqrt{\varphi_{s,t}}}
        \right).
    \end{align*}
\end{proof}

\section{A list of $q$-analogs of Lehmer series}

\subsection{Case $u=-t=\varphi_{s,t}\varphi_{s,t}^\prime$. Euler-Type 1}
From Eq.(\ref{eqn_dfg_iden1}),
\begin{equation}
    \sum_{n=0}^{\infty}\qbinom{2n}{n}_{q}\frac{(-1)^{n}(\varphi_{s,t}^{-3/2}\sqrt{-t}x)^n}{(-q;q)_{n}(-\sqrt{q};q)_{n}(1-\sqrt{q}^{2n-1})}
    =\frac{\sqrt{-t}}{(1-\sqrt{q})\sqrt{\varphi_{s,t}}}\frac{(-\varphi_{s,t}^{-1/2}x;q)_{\infty}}{(-\varphi_{s,t}^{-1/2}\sqrt{q}x;q)_{\infty}}.
\end{equation}
From Eq.(\ref{eqn_dfg_iden2}),
\begin{equation}
    \sum_{n=0}^{\infty}\qbinom{2n}{n}_{q}\frac{1}{(-\sqrt{q};q)_{n}(-q;q)_{n}}\left(\frac{-x}{\sqrt{-t\varphi_{s,t}}}\right)^n=\frac{(4x\sqrt{-q/t\varphi_{s,t}};q)_{\infty}}{(4x/\sqrt{-t\varphi_{s,t}};q)_{\infty}}.
\end{equation}
From Eq.(\ref{eqn_dfg_iden3}),
\begin{multline}
    \sum_{n=0}^{\infty}\qbinom{2n}{n}_{q}\frac{1}{(-\sqrt{q};q)_{n}(-q;q)_{n}(1-q^{n+1})}\left(\frac{4x}{\sqrt{-t\varphi_{s,t}}}\right)^n\\
    =\frac{1}{1-q}\cdot{}_{2}\phi_{1}\left(
        \begin{array}{c}
             \sqrt{q},q\\
             q^2
        \end{array};
        q,\frac{4x}{\sqrt{-t\varphi_{s,t}}}
        \right).
\end{multline}
From Eq.(\ref{eqn_dfg_iden4}),
\begin{multline}
    \sum_{n=1}^{\infty}\qbinom{2n}{n}_{q}\frac{1}{(-\sqrt{q};q)_{n}(-q;q)_{n}(1-q^n)}\left(\frac{4}{\sqrt{-t\varphi_{s,t}}}\right)^{n}x^{n}\\
    =\frac{4(1-\sqrt{q})x}{(1-q)^2\sqrt{-t\varphi_{s,t}}}{}_{3}\phi_{2}\left(
        \begin{array}{c}
             q\sqrt{q},q,q\\
             q^2,q^2
        \end{array};
        q,\frac{4tx}{t\sqrt{-t\varphi_{s,t}}}
        \right).
\end{multline}
From Eq.(\ref{eqn_dfg_iden5}),
\begin{multline}
    \sum_{n=1}^{\infty}\qbinom{2n}{n}_{q}\frac{1}{(-\sqrt{q};q)_{n}(-q;q)_{n}(1-q^{n})(1-q^{n+2})}\left(\frac{4x}{\sqrt{-t\varphi_{s,t}}}\right)^n\\
    =\frac{4(1-\sqrt{q})x}{(1-q)^2(1+q)\sqrt{-t\varphi_{s,t}}}{}_{3}\phi_{2}\left(
        \begin{array}{c}
             q\sqrt{q},q,q\\
             q^2,q^3
        \end{array};
        q,\frac{4x}{\sqrt{-t\varphi_{s,t}}}
        \right).
\end{multline}
From Eq.(\ref{eqn_dfg_iden6}),
\begin{multline}
    \sum_{n=1}^{\infty}\qbinom{2n}{n}_{q}\frac{(1-q^{n})}{(-\sqrt{q};q)_{n}(-q;q)_{n}}\left(\frac{4x}{\sqrt{-t\varphi_{s,t}}}\right)^n\\
    =\frac{4(1-\sqrt{q})x}{\sqrt{-t\varphi_{s,t}}}{}_{1}\phi_{0}\left(
        \begin{array}{c}
             q\sqrt{q}\\
             -
        \end{array};
        q,\frac{4x}{\sqrt{-t\varphi_{s,t}}}
        \right).
\end{multline}
From Eq.(\ref{eqn_dfg_iden7}),
\begin{multline}
    \sum_{n=1}^{\infty}\qbinom{2n}{n}_{q}\frac{(1-q^{n})^2(-u/t)^{\binom{n}{2}}}{(-\sqrt{q};q)_{n}(-q;q)_{n}}\left(\frac{4x}{\sqrt{-t\varphi_{s,t}}}\right)^n\\
    =\frac{4(1-\sqrt{q})x}{\sqrt{-t\varphi_{s,t}}}{}_{1}\phi_{0}\left(
        \begin{array}{c}
             q\sqrt{q}\\
             -
        \end{array};
        q,\frac{4qx}{\sqrt{-t\varphi_{s,t}}}
        \right)\\
        -\frac{4^2(\sqrt{q};q)_{2}x^2}{t\varphi_{s,t}}{}_{1}\phi_{0}\left(
        \begin{array}{c}
             q^2\sqrt{q}\\
             -
        \end{array};
        q,\frac{-4x}{\sqrt{-t\varphi_{s,t}}}
        \right).
\end{multline}
From Eq.(\ref{eqn_dfg_iden8}),
\begin{multline}
    \sum_{n=1}^{\infty}\qbinom{2n}{n}_{q}\frac{1}{(-\sqrt{q};q)_{n}(-q;q)_{n}(1-q^{2n+1})}\left(\frac{4x^2}{\sqrt{-t\varphi_{s,t}}}\right)^n\\
    =\frac{1}{1-q}{}_{3}\phi_{2}\left(
        \begin{array}{c}
             \sqrt{q},\sqrt{q},-\sqrt{q}\\
             q\sqrt{q},-q\sqrt{q}
        \end{array};
        q,\frac{-4x^2}{t\sqrt{\varphi_{s,t}}}
        \right).
\end{multline}

\subsection{Case $u=-tq=\varphi_{s,t}^{\prime2}$. Euler-type 2}

If $u=-tq$ and $\varphi_{s,t}^{-3/2}\sqrt{-t}x=q$, then
    \begin{multline*}
    \sum_{n=0}^{\infty}\qbinom{2n}{n}_{q}\frac{(-1)^{n}q^{\binom{n}{2}}}{(-q;q)_{n}(-\sqrt{q};q)_{n}(1-\sqrt{q}^{2n-1})}(\varphi_{s,t}^{-3/2}\sqrt{-t}x)^n\\
    =\frac{\sqrt{-t}(-\varphi_{s,t}^{-3/2}\sqrt{-t}x,\sqrt{q^{-1}};q)_{\infty}}{(1-\sqrt{q})\sqrt{\varphi_{s,t}}}{}_{2}\phi_{1}\left(
        \begin{array}{c}
             0,0\\
             -\varphi_{s,t}^{-3/2}\sqrt{-t}x
        \end{array};
        q,\sqrt{q^{-1}}
        \right).
\end{multline*}
If $u=-tq$ and $\varphi_{s,t}^{-3/2}\sqrt{-t}x=q$, then
    \begin{equation*}
    \sum_{n=0}^{\infty}\qbinom{2n}{n}_{q}\frac{(-1)^{n}q^{\binom{n+1}{2}}}{(-q;q)_{n}(-\sqrt{q};q)_{n}(1-\sqrt{q}^{2n-1})}
    =\frac{\sqrt{-t}}{(1-\sqrt{q})\sqrt{\varphi_{s,t}}}(-q;q)_{\infty}(-\sqrt{q};q^2)_{\infty}.
\end{equation*}
If $u=-tq$ and $4x=\sqrt{-t}q$, then
\begin{equation}
    \sum_{n=0}^{\infty}\qbinom{2n}{n}_{q}\frac{q^{\binom{n}{2}}}{(-\sqrt{q};q)_{n}(-q;q)_{n}}\left(\frac{-x}{\sqrt{-t\varphi_{s,t}}}\right)^n=(-q;q)_{\infty}(q\sqrt{q};q^2)_{\infty}.
\end{equation}
From Eq.(\ref{eqn_dfg_iden3}),
\begin{multline}
    \sum_{n=0}^{\infty}\qbinom{2n}{n}_{q}\frac{q^{\binom{n}{2}}}{(-\sqrt{q};q)_{n}(-q;q)_{n}(1-q^{n+1})}\left(\frac{4x}{\sqrt{-t\varphi_{s,t}}}\right)^n\\
    =\frac{1}{1-q}\cdot{}_{2}\phi_{2}\left(
        \begin{array}{c}
             \sqrt{q},q\\
             q^2,0
        \end{array};
        q,-\frac{4x}{\sqrt{-t\varphi_{s,t}}}
        \right).
\end{multline}
From Eq.(\ref{eqn_dfg_iden4}),
\begin{multline}
    \sum_{n=1}^{\infty}\qbinom{2n}{n}_{q}\frac{q^{\binom{n}{2}}}{(-\sqrt{q};q)_{n}(-q;q)_{n}(1-q^n)}\left(\frac{4}{\sqrt{-t\varphi_{s,t}}}\right)^{n}x^{n}\\
    =\frac{4(1-\sqrt{q})x}{(1-q)^2\sqrt{-t\varphi_{s,t}}}{}_{3}\phi_{3}\left(
        \begin{array}{c}
             q\sqrt{q},q,q\\
             q^2,q^2,0
        \end{array};
        q,-\frac{4qx}{\sqrt{-t\varphi_{s,t}}}
        \right).
\end{multline}
From Eq.(\ref{eqn_dfg_iden5}),
\begin{multline}
    \sum_{n=1}^{\infty}\qbinom{2n}{n}_{q}\frac{q^{\binom{n}{2}}}{(-\sqrt{q};q)_{n}(-q;q)_{n}(1-q^{n})(1-q^{n+2})}\left(\frac{4x}{\sqrt{-t\varphi_{s,t}}}\right)^n\\
    =\frac{4(1-\sqrt{q})x}{(1-q)^2(1+q)\sqrt{-t\varphi_{s,t}}}{}_{3}\phi_{3}\left(
        \begin{array}{c}
             q\sqrt{q},q,q\\
             q^2,q^3,0
        \end{array};
        q,-\frac{4q\theta}{\sqrt{-t\varphi_{s,t}}}
        \right).
\end{multline}
From Eq.(\ref{eqn_dfg_iden6}),
\begin{multline}
    \sum_{n=1}^{\infty}\qbinom{2n}{n}_{q}\frac{(1-q^{n})q^{\binom{n}{2}}}{(-\sqrt{q};q)_{n}(-q;q)_{n}}\left(\frac{4x}{\sqrt{-t\varphi_{s,t}}}\right)^n\\
    =\frac{4(1-\sqrt{q})x}{\sqrt{-t\varphi_{s,t}}}{}_{1}\phi_{1}\left(
        \begin{array}{c}
             q\sqrt{q}\\
             -
        \end{array};
        q,-\frac{4qx}{\sqrt{-t\varphi_{s,t}}}
        \right).
\end{multline}
From Eq.(\ref{eqn_dfg_iden7}),
\begin{multline}
    \sum_{n=1}^{\infty}\qbinom{2n}{n}_{q}\frac{(1-q^{n})^2q^{\binom{n}{2}}}{(-\sqrt{q};q)_{n}(-q;q)_{n}}\left(\frac{4x}{\sqrt{-t\varphi_{s,t}}}\right)^n\\
    =\frac{4(1-\sqrt{q})x}{\sqrt{-t\varphi_{s,t}}}{}_{1}\phi_{1}\left(
        \begin{array}{c}
             q\sqrt{q}\\
             -
        \end{array};
        q,-\frac{4q^2x}{\sqrt{-t\varphi_{s,t}}}
        \right)\\
        -\frac{4^2(\sqrt{q};q)_{2}qx^2}{t\varphi_{s,t}}{}_{1}\phi_{0}\left(
        \begin{array}{c}
             q^2\sqrt{q}\\
             -
        \end{array};
        q,\frac{-4q^2x}{\sqrt{-t\varphi_{s,t}}}
        \right).
\end{multline}
From Eq.(\ref{eqn_dfg_iden8}),
\begin{multline}
    \sum_{n=1}^{\infty}\qbinom{2n}{n}_{q}\frac{q^{\binom{n}{2}}}{(-\sqrt{q};q)_{n}(-q;q)_{n}(1-q^{2n+1})}\left(\frac{4x^2}{\sqrt{-t\varphi_{s,t}}}\right)^n\\
    =\frac{1}{1-q}{}_{3}\phi_{3}\left(
        \begin{array}{c}
             \sqrt{q},\sqrt{q},-\sqrt{q}\\
             q\sqrt{q},-q\sqrt{q}
        \end{array};
        q,\frac{-4x^2}{t\sqrt{\varphi_{s,t}}}
        \right).
\end{multline}

\subsection{Case $u=-tq^2=\varphi_{s,t}^{\prime3}/\varphi_{s,t}$, $x\mapsto qx$. Rogers-Ramanujan-type}

From Eq.(\ref{eqn_dfg_iden1}),
\begin{multline}
    \sum_{n=0}^{\infty}\qbinom{2n}{n}_{q}\frac{(-1)^{n}q^{n^2}}{(-q;q)_{n}(-\sqrt{q};q)_{n}(1-\sqrt{q}^{2n-1})}(\varphi_{s,t}^{-3/2}\sqrt{-t}x)^n\\
    =\frac{\sqrt{-t}}{(1-\sqrt{q})\sqrt{\varphi_{s,t}}}{}_{1}\phi_{2}\left(
        \begin{array}{c}
             \sqrt{q^{-1}}\\
             0,0
        \end{array};
        q,-\varphi_{s,t}^{-1/2}q\sqrt{q}x
        \right).
\end{multline}    
From Eq.(\ref{eqn_dfg_iden2}),
\begin{equation}
    \sum_{n=0}^{\infty}\qbinom{2n}{n}_{q}\frac{q^{n^2}}{(-\sqrt{q};q)_{n}(-q;q)_{n}}\left(\frac{4x}{\sqrt{-t\varphi_{s,t}}}\right)^n\\
    ={}_{1}\phi_{2}\left(
        \begin{array}{c}
             \sqrt{q}\\
             0,0
        \end{array};
        q,\frac{4qx}{\sqrt{-t\varphi_{s,t}}}
        \right).
\end{equation}
From Eq.(\ref{eqn_dfg_iden3}),
\begin{multline}
    \sum_{n=0}^{\infty}\qbinom{2n}{n}_{q}\frac{q^{n^2}}{(-\sqrt{q};q)_{n}(-q;q)_{n}(1-q^{n+1})}\left(\frac{4x}{\sqrt{-t\varphi_{s,t}}}\right)^n\\
    =\frac{1}{1-q}\cdot{}_{2}\phi_{3}\left(
        \begin{array}{c}
             \sqrt{q},q\\
             q^2,0,0
        \end{array};
        q,\frac{4qx}{\sqrt{-t\varphi_{s,t}}}
        \right).
\end{multline}
From Eq.(\ref{eqn_dfg_iden4}),
\begin{multline}
    \sum_{n=1}^{\infty}\qbinom{2n}{n}_{q}\frac{q^{n^2}}{(-\sqrt{q};q)_{n}(-q;q)_{n}(1-q^n)}\left(\frac{4}{\sqrt{-t\varphi_{s,t}}}\right)^{n}x^{n}\\
    =\frac{4(1-\sqrt{q})qx}{(1-q)^2\sqrt{-t\varphi_{s,t}}}{}_{3}\phi_{4}\left(
        \begin{array}{c}
             q\sqrt{q},q,q\\
             q^2,q^2,0,0
        \end{array};
        q,\frac{4q^3x}{\sqrt{-t\varphi_{s,t}}}
        \right).
\end{multline}
From Eq.(\ref{eqn_dfg_iden5}),
\begin{multline}
    \sum_{n=1}^{\infty}\qbinom{2n}{n}_{q}\frac{q^{n^2}}{(-\sqrt{q};q)_{n}(-q;q)_{n}(1-q^{n})(1-q^{n+2})}\left(\frac{4x}{\sqrt{-t\varphi_{s,t}}}\right)^n\\
    =\frac{4(1-\sqrt{q})qx}{(1-q)^2(1+q)\sqrt{-t\varphi_{s,t}}}{}_{3}\phi_{4}\left(
        \begin{array}{c}
             q\sqrt{q},q,q\\
             q^2,q^3,0,0
        \end{array};
        q,\frac{4q^3x}{\sqrt{-t\varphi_{s,t}}}
        \right).
\end{multline}
From Eq.(\ref{eqn_dfg_iden6}),
\begin{multline}
    \sum_{n=1}^{\infty}\qbinom{2n}{n}_{q}\frac{(1-q^{n})q^{n^2}}{(-\sqrt{q};q)_{n}(-q;q)_{n}}\left(\frac{4x}{\sqrt{-t\varphi_{s,t}}}\right)^n\\
    =\frac{4(1-\sqrt{q})qx}{\sqrt{-t\varphi_{s,t}}}{}_{1}\phi_{2}\left(
        \begin{array}{c}
             q\sqrt{q}\\
             0,0
        \end{array};
        q,\frac{4q^3x}{\sqrt{-t\varphi_{s,t}}}
        \right).
\end{multline}
From Eq.(\ref{eqn_dfg_iden7}),
\begin{multline}
    \sum_{n=1}^{\infty}\qbinom{2n}{n}_{q}\frac{(1-q^{n})^2q^{n^2}}{(-\sqrt{q};q)_{n}(-q;q)_{n}}\left(\frac{4x}{\sqrt{-t\varphi_{s,t}}}\right)^n\\
    =\frac{4(1-\sqrt{q})qx}{\sqrt{-t\varphi_{s,t}}}{}_{1}\phi_{2}\left(
        \begin{array}{c}
             q\sqrt{q}\\
             0,0
        \end{array};
        q,\frac{4q^3x}{\sqrt{-t\varphi_{s,t}}}
        \right)\\
        -\frac{4^2(\sqrt{q};q)_{2}q^4x^2}{t\varphi_{s,t}}{}_{1}\phi_{2}\left(
        \begin{array}{c}
             q^2\sqrt{q}\\
             0,0
        \end{array};
        q,\frac{4q^5x}{\sqrt{-t\varphi_{s,t}}}
        \right).
\end{multline}
From Eq.(\ref{eqn_dfg_iden8}),
\begin{multline}
    \sum_{n=1}^{\infty}\qbinom{2n}{n}_{q}\frac{q^{n^2+n+1}}{(-\sqrt{q};q)_{n}(-q;q)_{n}(1-q^{2n+1})}\left(\frac{4x^2}{\sqrt{-t\varphi_{s,t}}}\right)^n\\
    =\frac{1}{1-q}{}_{3}\phi_{4}\left(
        \begin{array}{c}
             \sqrt{q},\sqrt{q},-\sqrt{q}\\
             q\sqrt{q},-q\sqrt{q},0,0
        \end{array};
        q,\frac{-4q^2x^2}{t\sqrt{\varphi_{s,t}}}
        \right).
\end{multline}

\subsection{Case $u=-t\sqrt{q}=\varphi_{s,t}^{1/2}\varphi_{s,t}^{\prime3/2}$. Exton-type.}
From Eq.(\ref{eqn_dfg_iden1}),
\begin{multline}
    \sum_{n=0}^{\infty}\qbinom{2n}{n}_{q}\frac{(-1)^{n}q^{\frac{1}{2}\binom{n}{2}}}{(-q;q)_{n}(-\sqrt{q};q)_{n}(1-\sqrt{q}^{2n-1})}(\varphi_{s,t}^{-3/2}\sqrt{-t}x)^n\\
    =\frac{\sqrt{-t}}{(1-\sqrt{q})\sqrt{\varphi_{s,t}}}{}_{2}\phi_{2}\left(
        \begin{array}{c}
             \sqrt[4]{q^{-1}},-\sqrt[4]{q^{-1}}\\
             -\sqrt{q},0
        \end{array};
        \sqrt{q},-\varphi_{s,t}^{-1/2}\sqrt{q}x
        \right).
\end{multline}    
From Eq.(\ref{eqn_dfg_iden2}),
\begin{equation}
    \sum_{n=0}^{\infty}\qbinom{2n}{n}_{q}\frac{q^{\frac{1}{2}\binom{n}{2}}}{(-\sqrt{q};q)_{n}(-q;q)_{n}}\left(\frac{4x}{\sqrt{-t\varphi_{s,t}}}\right)^n\\
    ={}_{2}\phi_{2}\left(
        \begin{array}{c}
             \sqrt[4]{q},-\sqrt[4]{q}\\
             -\sqrt{q},0
        \end{array};
        \sqrt{q},\frac{4x}{\sqrt{-t\varphi_{s,t}}}
        \right).
\end{equation}
From Eq.(\ref{eqn_dfg_iden3}),
\begin{multline}
    \sum_{n=0}^{\infty}\qbinom{2n}{n}_{q}\frac{q^{\frac{1}{2}\binom{n}{2}}}{(-\sqrt{q};q)_{n}(-q;q)_{n}(1-q^{n+1})}\left(\frac{4x}{\sqrt{-t\varphi_{s,t}}}\right)^n\\
    =\frac{1}{1-q}\cdot{}_{4}\phi_{4}\left(
        \begin{array}{c}
             \sqrt[4]{q},-\sqrt[4]{q},\sqrt{q},-\sqrt{q}\\
             q,-q,-\sqrt{q},0
        \end{array};
        \sqrt{q},\frac{4x}{\sqrt{-t\varphi_{s,t}}}
        \right).
\end{multline}
From Eq.(\ref{eqn_dfg_iden4}),
\begin{multline}
    \sum_{n=1}^{\infty}\qbinom{2n}{n}_{q}\frac{q^{\frac{1}{2}\binom{n}{2}}}{(-\sqrt{q};q)_{n}(-q;q)_{n}(1-q^n)}\left(\frac{4}{\sqrt{-t\varphi_{s,t}}}\right)^{n}x^{n}\\
    =\frac{4(1-\sqrt{q})x}{(1-q)^2\sqrt{-t\varphi_{s,t}}}{}_{6}\phi_{6}\left(
        \begin{array}{c}
             \sqrt[4]{q^{3}},-\sqrt[4]{q^{3}},\sqrt{q},-\sqrt{q},\sqrt{q},-\sqrt{q}\\
             q,-q,q,-q,-\sqrt{q},0
        \end{array};
        \sqrt{q},\frac{4\sqrt{q}x}{\sqrt{-t\varphi_{s,t}}}
        \right).
\end{multline}
From Eq.(\ref{eqn_dfg_iden5}),
\begin{multline}
    \sum_{n=1}^{\infty}\qbinom{2n}{n}_{q}\frac{q^{\frac{1}{2}\binom{n}{2}}}{(-\sqrt{q};q)_{n}(-q;q)_{n}(1-q^{n})(1-q^{n+2})}\left(\frac{4x}{\sqrt{-t\varphi_{s,t}}}\right)^n\\
    =\frac{4(1-\sqrt{q})x}{(1-q)^2(1+q)\sqrt{-t\varphi_{s,t}}}{}_{6}\phi_{6}\left(
        \begin{array}{c}
             \sqrt[4]{q^3},-\sqrt[4]{q^3},\sqrt{q},-\sqrt{q},\sqrt{q},-\sqrt{q}\\
             q,-q,\sqrt{q^3},\sqrt{q^3},-\sqrt{q},0
        \end{array};
        \sqrt{q},\frac{4\sqrt{q}x}{\sqrt{-t\varphi_{s,t}}}
        \right).
\end{multline}
From Eq.(\ref{eqn_dfg_iden6}),
\begin{multline}
    \sum_{n=1}^{\infty}\qbinom{2n}{n}_{q}\frac{(1-q^{n})q^{\frac{1}{2}\binom{n}{2}}}{(-\sqrt{q};q)_{n}(-q;q)_{n}}\left(\frac{4x}{\sqrt{-t\varphi_{s,t}}}\right)^n\\
    =\frac{4(1-\sqrt{q})x}{\sqrt{-t\varphi_{s,t}}}{}_{2}\phi_{2}\left(
        \begin{array}{c}
             \sqrt[4]{q^3},-\sqrt[4]{q^3}\\
             -\sqrt{q},0
        \end{array};
        \sqrt{q},\frac{4\sqrt{q}x}{\sqrt{-t\varphi_{s,t}}}
        \right).
\end{multline}
From Eq.(\ref{eqn_dfg_iden7}),
\begin{multline}
    \sum_{n=1}^{\infty}\qbinom{2n}{n}_{q}\frac{(1-q^{n})^2q^{\frac{1}{2}\binom{n}{2}}}{(-\sqrt{q};q)_{n}(-q;q)_{n}}\left(\frac{4x}{\sqrt{-t\varphi_{s,t}}}\right)^n\\
    =\frac{4(1-\sqrt{q})x}{\sqrt{-t\varphi_{s,t}}}{}_{2}\phi_{2}\left(
        \begin{array}{c}
             \sqrt[4]{q^3},-\sqrt[4]{q^3}\\
             -\sqrt{q},0
        \end{array};
        \sqrt{q},\frac{-4q\sqrt{q}x}{\sqrt{-t\varphi_{s,t}}}
        \right)\\
        -\frac{4^2(\sqrt{q};q)_{2}\sqrt{q}x^2}{t\varphi_{s,t}}{}_{2}\phi_{2}\left(
        \begin{array}{c}
             q\sqrt[4]{q},-q\sqrt[4]{q}\\
             -\sqrt{q},0
        \end{array};
        \sqrt{q},\frac{-4qx}{\sqrt{-t\varphi_{s,t}}}
        \right).
\end{multline}
From Eq.(\ref{eqn_dfg_iden8}),
\begin{multline}
    \sum_{n=1}^{\infty}\qbinom{2n}{n}_{q}\frac{q^{\frac{1}{2}\binom{n}{2}}}{(-\sqrt{q};q)_{n}(-q;q)_{n}(1-q^{2n+1})}\left(\frac{4x^2}{\sqrt{-t\varphi_{s,t}}}\right)^n\\
    =\frac{1}{1-q}{}_{6}\phi_{6}\left(
        \begin{array}{c}
             \sqrt[4]{q},-\sqrt[4]{q},\sqrt[4]{q},-\sqrt[4]{q},i\sqrt[4]{q},-i\sqrt[4]{q}\\
             \sqrt[4]{q^3},-\sqrt[4]{q^3},i\sqrt[4]{q^3},-i\sqrt[4]{q^3},-\sqrt{q},0
        \end{array};
        \sqrt{q},\frac{-4x^2}{t\sqrt{\varphi_{s,t}}}
        \right).
\end{multline}
where $i=\sqrt{-1}$.

\subsection{Cases $u=-t\varphi_{s,t}^2=\varphi_{s,t}^3\varphi_{s,t}^\prime$.}

The $(s,t)$-analogue of the Catalan numbers is 
\begin{equation}
    C_{\{n\}}=\fibonomial{2n}{n}_{s,t}\frac{1}{\brk[c]{n+1}_{s,t}}
\end{equation}
and in form deformed $q$-analogue
\begin{equation}
    C_{\{n\}}=\qbinom{2n}{n}_{q}\frac{\varphi_{s,t}^{2\binom{n}{2}}(1-q)}{1-q^{n+1}}.
\end{equation}

From Eq.(\ref{eqn_dfg_iden3}),
\begin{equation}
    \sum_{n=0}^{\infty}\frac{C_{\{n\}}}{(-\sqrt{q};q)_{n}(-q;q)_{n}}\left(\frac{4x}{\sqrt{-t\varphi_{s,t}}}\right)^n
    ={}_{2}\Phi_{1}\left(
        \begin{array}{c}
             \sqrt{q},q\\
             q^2
        \end{array};
        q,\varphi_{s,t}^2,\frac{4x}{\sqrt{-t\varphi_{s,t}}}
        \right).
\end{equation}
From Eq.(\ref{eqn_dfg_iden4}),
\begin{multline}
    \sum_{n=1}^{\infty}\frac{C_{\{n\}}(1-q^{n+1})}{(-\sqrt{q};q)_{n}(-q;q)_{n}(1-q^n)}\left(\frac{4}{\sqrt{-t\varphi_{s,t}}}\right)^{n}x^{n}\\
    =\frac{4(1-\sqrt{q})x}{(1-q)\sqrt{-t\varphi_{s,t}}}{}_{3}\Phi_{2}\left(
        \begin{array}{c}
             q\sqrt{q},q,q\\
             q^2,q^2
        \end{array};
        q,\varphi_{s,t}^2,\frac{4\varphi_{s,t}^2x}{\sqrt{-t\varphi_{s,t}}}
        \right).
\end{multline}
From Eq.(\ref{eqn_dfg_iden5}),
\begin{multline}
    \sum_{n=1}^{\infty}\frac{C_{\{n\}}}{(-\sqrt{q};q)_{n}(-q;q)_{n}(1-q^{n})}\left(\frac{4x}{\sqrt{-t\varphi_{s,t}}}\right)^n\\
    =\frac{4(1-\sqrt{q})x}{(1-q)(1+q)\sqrt{-t\varphi_{s,t}}}{}_{3}\Phi_{2}\left(
        \begin{array}{c}
             q\sqrt{q},q,q\\
             q^2,q^3
        \end{array};
        q,\varphi_{s,t}^2,\frac{4\varphi_{s,t}^2x}{\sqrt{-t\varphi_{s,t}}}
        \right).
\end{multline}
From Eq.(\ref{eqn_dfg_iden6}),
\begin{multline}
    \sum_{n=1}^{\infty}C_{\{n\}}\frac{(1-q^{n})(1-q^{n+1})}{(-\sqrt{q};q)_{n}(-q;q)_{n}}\left(\frac{4x}{\sqrt{-t\varphi_{s,t}}}\right)^n\\
    =\frac{4(1-q)(1-\sqrt{q})x}{\sqrt{-t\varphi_{s,t}}}{}_{1}\Phi_{0}\left(
        \begin{array}{c}
             q\sqrt{q}\\
             -
        \end{array};
        q,\varphi_{s,t}^2,\frac{4\varphi_{s,t}^2x}{\sqrt{-t\varphi_{s,t}}}
        \right).
\end{multline}
From Eq.(\ref{eqn_dfg_iden7}),
\begin{multline}
    \sum_{n=1}^{\infty}C_{\{n\}}\frac{(1-q^{n})^2(1-q^{n+1})}{(-\sqrt{q};q)_{n}(-q;q)_{n}}\left(\frac{4x}{\sqrt{-t\varphi_{s,t}}}\right)^n\\
    =\frac{4(1-q)(1-\sqrt{q})x}{\sqrt{-t\varphi_{s,t}}}{}_{1}\Phi_{0}\left(
        \begin{array}{c}
             q\sqrt{q}\\
             -
        \end{array};
        q,\varphi_{s,t}^2,\frac{4q\varphi_{s,t}^2x}{\sqrt{-t\varphi_{s,t}}}
        \right)\\
        -\frac{4^2(1-q)(\sqrt{q};q)_{2}\varphi_{s,t}^2x^2}{t\varphi_{s,t}}{}_{1}\Phi_{0}\left(
        \begin{array}{c}
             q^2\sqrt{q}\\
             -
        \end{array};
        q,\varphi_{s,t}^2,\frac{-4\varphi_{s,t}^4x}{\sqrt{-t\varphi_{s,t}}}
        \right).
\end{multline}
From Eq.(\ref{eqn_dfg_iden8}),
\begin{multline}
    \sum_{n=1}^{\infty}C_{\{n\}}\frac{(1-q^{n+1})}{(-\sqrt{q};q)_{n}(-q;q)_{n}(1-q^{2n+1})}\left(\frac{4x^2}{\sqrt{-t\varphi_{s,t}}}\right)^n\\
    ={}_{3}\Phi_{2}\left(
        \begin{array}{c}
             \sqrt{q},\sqrt{q},-\sqrt{q}\\
             q\sqrt{q},-q\sqrt{q}
        \end{array};
        q,\varphi_{s,t}^2,\frac{-4x^2}{t\sqrt{\varphi_{s,t}}}
        \right).
\end{multline}

\end{document}